\documentclass[11pt,reqno]{amsart}
\usepackage{amssymb,amsfonts,amsthm,amscd,stmaryrd,dsfont,esint,upgreek,constants,todonotes}
\usepackage[left=1.3in, right=1.3in, marginparwidth=2in]{geometry} 

\usepackage{amsmath,amsthm,verbatim,amssymb,amsfonts,amscd,graphicx}
\usepackage[shortlabels]{enumitem}
\usepackage{mathtools}
\usepackage{xcolor}
\usepackage{hyperref}
\hypersetup{colorlinks, citecolor=violet, linkcolor=blue, urlcolor=Blue}

\newtheorem{thm}{Theorem}[section]

\newtheorem{prop}[thm]{Proposition}
\newtheorem{lem}[thm]{Lemma}
\theoremstyle{definition}
\newtheorem{assumption}[thm]{Assumption}

\newtheorem{remark}[thm]{Remark}

\newcommand{\pr}{\mathbb{P}}

\newcommand{\Q}{\mathbb{Q}}
\newcommand{\R}{\mathbb{R}}

\newcommand{\E}{\mathbb{E}}

\newcommand{\supp}{\text{supp}}

\newcommand{\sP}{\mathcal{P}}
\newcommand{\sL}{\mathcal{L}}
\newcommand{\sA}{\mathcal{A}}

\newcommand{\bx}{\bar{x}}
\newcommand{\bz}{\bar{z}}
\newcommand{\bF}{\mathbb{F}}
\newcommand{\sF}{\mathcal{F}}

\newcommand{\linf}{L^{\infty}}
\newcommand{\sinf}{\mathcal{S}^{\infty}}

\newcommand{\calpha}{C^{\alpha}}
\newcommand{\lpee}{L^{p}}
\newcommand{\spee}{\mathcal{S}^p}
\newcommand{\ltwo}{L^2}

  \DeclarePairedDelimiter\norm{\lVert}{\rVert}  

\begin{document}

\title[Monotonicity condition for ergodic BSDEs]{A new monotonicity condition for ergodic BSDEs and ergodic control with super-quadratic Hamiltonians}
\author{Joe Jackson}
\address{Department of Mathematics, The University of Texas at Austin}
\email{jjackso1@utexas.edu}%

\author{Gechun Liang}
\address{Department of Statistics, University of Warwick}
\email{g.liang@warwick.ac.uk}

\thanks{During the preparation of this work the first author has been supported by the National Science Foundation under Grant No. DGE1610403 (2020-2023). Any opinions, findings and conclusions or recommendations expressed in this material are those of the author(s) and do not necessarily reflect the views of the National Science Foundation (NSF).
}





\begin{abstract}
    We establish the existence (and in an appropriate sense uniqueness) of Markovian solutions for ergodic BSDEs under a novel monotonicity condition. Our monotonicity condition allows us to prove existence even when the driver $f$ has arbitrary (in particular super-quadratic) growth in $z$, which reveals an interesting trade-off between monotonicity and growth for ergodic BSDEs. The technique of proof is to establish a probabilistic representation of the derivative of the Markovian solution, and then use this representation to obtain a-priori estimates. Our study is motivated by applications to ergodic control, and we use our existence result to prove the existence of optimal controls for a class of ergodic control problems with potentially super-quadratic Hamiltonians. We also treat a class of drivers coming from the construction of forward performance processes, and interpret our monotonicity condition in this setting. 
\end{abstract}

\maketitle

\section{Introduction} 

\subsection{Ergodic BSDEs}
In this paper, we study the ergodic backward stochastic differential equation (EBSDE)
\begin{align} \label{ebsdeintro}
    \begin{cases} dX_t = b(X_t) dt + dW_t, \\
    dY_t = -f(X_t,Z_t) dt + \lambda dt + Z_t \cdot dW_t. 
    \end{cases}
\end{align}
Here $W$ is a $d$-dimensional Brownian motion, and the data consists of a vector field $b : \R^d \to \R^d$ and a function $f : \R^d \times \R^d \to \R^d$ called the driver. We are interested in Markovian solutions, i.e. triples $(u,v,\lambda)$ where $u: \R^d \to \R$, $v : \R^d \to \R^d$, and $\lambda \in \R$, such that for each $x \in \R^d$ and $0 \leq t \leq T < \infty$, we have 
\begin{align*}
    Y_t^x = Y_T^x + \int_t^T f(X_s^x,Z_s^x) ds - \lambda(T-t) - \int_t^T Z_s^x \cdot dW_s, 
\end{align*}
where $X^x$ solves $X^x_t = x + \int_0^t b(X^x_s) ds + W_t$ and $(Y^x,Z^x) = (u(X^x),v(X^x))$. We recall that there is a natural connection between the EBSDE \eqref{ebsdeintro} and the partial differential equation (PDE) 
\begin{align} \label{pdeintro}
    \frac{1}{2} \Delta u + b \cdot \nabla u +  f(x,\nabla u) = \lambda \text{ in }\R^d, 
\end{align}
in the sense that if there is a pair $(u,\lambda)$ with $u : \R^d \to \R$ and $\lambda \in \R$ satisfying \eqref{pdeintro}, then It\^o's formula reveals that $(u,\nabla u, \lambda)$ is a Markovian solution to the EBSDE \eqref{ebsdeintro}. 

\subsection{Related literature}

\subsubsection{Probabilistic approaches}
Ergodic BSDEs were first introduced in \cite{Fuhrman2009ErgodicBA} in an infinite dimensional setting. Applied to the somewhat simpler equation \eqref{ebsdeintro}, the results of \cite{Fuhrman2009ErgodicBA} show existence and uniqueness of a Markovian solution when the drift $b$ is dissipative, i.e. 
\begin{align} \label{dissipativeintro}
    (b(x) - b(\bx)) \cdot (x - \bx) \leq - \delta |x - \bx|^2
\end{align}
for some $\delta > 0$, and the driver $f$ is Lipschitz. In \cite{Debussche2010ErgodicBU}, well-posedness is established under weaker conditions on $b$, but still with the Lipschitz condition on $f$. In \cite{Hu2015APA}, ergodic BSDEs are used to study the long-time behavior of parabolic PDEs, still in an infinite-dimensional framework. In the finite-dimensional setting, \cite{richou2008} and \cite{COSSO20161932} consider generalizations of the ergodic BSDE \eqref{ebsdeintro} involving reflection and jumps, respectively. In \cite{Liang2017RepresentationOH}, \cite{chonghuliang}, and \cite{huliangtang}  a connection is developed between EBSDEs and forward performance processes. In \cite{Liang2017RepresentationOH}, the state space is finite-dimensional, and the driver $f$ is of quadratic growth in $z$ and satisfies the regularity condition
\begin{align}
    |f(x,z) - f(\bx,z)| \leq C_x(1 + |z|)|x - \bx|.
\end{align}
In \cite{Liang2017RepresentationOH}, existence of a Markovian solution is established for \eqref{ebsdeintro} under the condition $\delta > C_x$. Similar ``sufficiently dissipative" conditions appear in \cite{huliangtang} and \cite{chonghuliang}. Finally, we mention that super-quadratic BSDEs have been studied on a finite horizon, most notable in \cite{delbaenhubao} where a number of results, both positive and negative were obtained.

\subsubsection{Analytical approaches}
The works listed above study the EBSDE \eqref{ebsdeintro} through probabilistic methods. There is also a large literature dealing directly with the associated PDE \eqref{pdeintro} through both probabilistic and analytical techniques. Here the focus is not only on well-posedness of the equation \eqref{pdeintro}, but also on its connection to the behavior of the parabolic problem
\begin{align}
    \begin{cases}
    \partial_t u + \frac{1}{2} \Delta u + b \cdot \nabla u + f(x, \nabla u) = 0 \text{ in }(0,T) \times \R^d, \\
    u(T,x) = g(x), \, \, x \in \R^d
    \end{cases}
\end{align}
as $T \to \infty$. The works \cite{Bensoussan1992}, \cite{Ichihara2013LargeTB}, \cite{Kaise2006ONTS}, and \cite{fujita}  all contain well-posedness results for the PDE \eqref{pdeintro} when the driver $f$ has quadratic growth in $z$. But typically the PDE approach requires strict convexity\footnote{More precisely, it is typical in the PDE literature to consider the equation $- \frac{1}{2} \Delta v - b \cdot \nabla v + H(x,Dv) = \gamma$ with $H = H(x,z)$ strictly convex in $z$. Make the transformation $u = -v$, $\lambda = - \gamma$, $f(x,z) = -H(x,-z)$ to see that this is equivalent to \eqref{pdeintro} with $f$ strictly convex in $z$.} of $f$ in $z$, see e.g. \cite{Ichihara2013LargeTB}. The recent work \cite{chasseigneichihara} contains some remarkable results on the qualitative behavior of the discounted analogue of \eqref{pdeintro} with super-quadratic Hamiltonian $\frac{1}{m} |\nabla u|^m$ for $m > 2$, but the focus there is on qualitative behavior rather than general well-posedness results.

\subsection{Our results}

\subsubsection{The main existence result}

The purpose of the present paper is to study the well-posedness of the ergodic BSDE \eqref{ebsdeintro} when $f$ is just locally Lipschitz, and in particular can have arbitrary growth in $z$. Our main result is Theorem \ref{thm:main}, which gives the existence (and in an appropriate sense uniqueness) of a Markovian solution to \eqref{ebsdeintro} when $f$ splits as $f = g + h$, and the data $(b, g, h)$ is such that 
\begin{itemize}
    \item $b$ is Lipschitz and dissipative, 
    \item $g = g(x,z)$ satisfies
    \begin{itemize}
        \item $\sup_{x \in \R^d} |g(x,0)| < \infty$,
        \item $g$ is Lipschitz (in both arguments) on $\R^d \times B_R$ for each $R > 0$
        \item the map $z \mapsto \partial_x g(x,z)$ is monotone, in the sense that $\big(\partial_x g(x,z) - \partial_x g(x, \bz)\big) \cdot (z - \bz) \leq 0$. 
    \end{itemize}
    \item $h = h(x,z)$ satisfies 
    \begin{itemize}
        \item $\sup_{x \in \R^d} |h(x,0)| < \infty$
        \item $x \mapsto h(x,z)$ is Lipschitz, uniformly in $z$
        \item $h$ is Lipschitz (in both arguments) on $\R^d \times B_R$, for each $R > 0$
    \end{itemize}
\end{itemize}
To the best of our knowledge, Theorem \ref{thm:main} is the first existence result for ergodic BSDEs which allows the driver to grow arbitrarily fast in $z$. Moreover, we emphasize that there is no requirement on the dissipativity constant $\delta$ verifying $(b(x) - b(\bx))\cdot (x - \bx) \leq - \delta |x - \bx|^2$ beyond positivity. Thus, while we do not handle the ``weak dissipative" case treated in \cite{Debussche2010ErgodicBU}, neither do we require that $\delta$ be ``large enough" as in \cite{Liang2017RepresentationOH}. We also note that we have established in the Appendix a rigorous connection between the EBSDE \eqref{ebsdeintro} and the PDE \eqref{pdeintro}, which allows us to assert that if $(u,v,\lambda)$ is the unique Markovian solution to \eqref{ebsdeintro}, then in fact $u$ is a classical solution of the PDE \eqref{pdeintro} and $v = \nabla u$ a.e.

We remark that Theorem \ref{thm:main} can be extended to the case where $dX_t = b(X_t) dt + \kappa dW_t$ for an appropriate $\kappa \in \R^{d \times d}$ in a relatively straightforward manner. In this case, the relevant monotonicity condition becomes $\big(\partial_x g(x,\kappa z) - \partial_x g(x, \kappa z)\big) \cdot (z - \bz) \leq 0$, i.e. we need monotonicity of the map $z \mapsto \partial_x g(x,\kappa z)$ rather than $z \mapsto \partial_x g(x,z)$. We do not include this slight generalization here for the sake of notational simplicity.

\subsubsection{Discussion of the monotonicity condition}
The monotonicity condition
\begin{align} \label{monotoneintro}
\big(\partial_x g(x,z) - \partial_x g(x, \bz)\big) \cdot (z - \bz) \leq 0
\end{align}
is the main novelty of the paper, and it is worth discussing it in detail. First, recall that when studying the corresponding PDE \eqref{pdeintro}, it is common to assume strict convexity of $f$ in $z$, which amounts to the requirement that 
\begin{align}
    \big(\partial_z f(x,z) - \partial_z f(x,\bz)\big) \cdot (z - \bz) > C |z - \bz|^2
\end{align}
for some $C > 0$. Thus the existing PDE approaches require a strict monotonicity condition on $z \mapsto \partial_z f(x,z)$, while our main condition is instead the monotonicity of $z \mapsto \partial_x g(x,z)$. 
As for the probabilistic perspective, we note that monotonicity conditions are fairly common in the BSDE literature. For example, for infinite horizon BSDEs of the form \begin{align*}
    dY_t = - f(Y_t,Z_t) dt + Z_t \cdot dW_t, 
\end{align*}
the requirement that $(f(y,z) - f(\bar{y},z)) \cdot (y - \bar{y}) \leq - \delta |y - \bar{y}|^2$ is standard, see \cite{BRIAND1998455}. For FBSDEs of the type 
\begin{align*}
    \begin{cases} 
    dX_t = b(X_t, Y_t, Z_t) dt + \sigma(X_t,Y_t,Z_t)  dW_t, \\
    dY_t = - f(X_t, Y_t, Z_t) dt + Z_t \cdot dW_t, \,\, Y_T = g(X_T), 
    \end{cases}
\end{align*}
a joint monotonicity property of the maps $b,\sigma, f$ is often imposed, see Section 8.4 of \cite{Zhang} and the references therein. But the monotonicity condition considered here is unique from probabilistic perspective in that it really involves monotonicity of a certain \textit{derivative} of $f$, rather than of $f$ itself.

Here are some examples of drivers $f$ which satisfy the hypotheses of our main existence result: 
\begin{itemize}
    \item If $f$ is globally Lipschitz and $f(x,0)$ is bounded, as in \cite{Fuhrman2009ErgodicBA} and \cite{Debussche2010ErgodicBU}, then it satisfies our condition (take $g = 0$). 
    \item $f(x,z) = g(z) + h(x)$, where $g$ is locally Lipschitz and $h$ is Lipschitz and bounded. 
    \item $f(x,z) = g(x,z) = g_1(x)g_2(z)$, where $d = 1$ and $g_1 : \R \to \R$ is bounded, Lipschitz, and non-increasing (non-decreasing), $g_2 : \R \to \R$ is locally Lipschitz and non-decreasing (non-increasing)
\end{itemize}

\subsubsection{Strategy of the proof}
The main step in proving Theorem \ref{thm:main} is to establish an a-priori estimate of the Lipschitz constant of $u$, where $(u,v,\lambda)$ is the Markovian solution of \eqref{ebsdeintro}. This is done with a mix of probabilistic and analytical techniques. More specifically, we first differentiate the PDE \eqref{pdeintro} which we expect $u$ to solve in order to find a PDE system for $\nabla u$. Then, we represent $\nabla u$ in terms of a certain system of infinite-horizon BSDEs, and finally we prove $\linf$ estimates for this BSDE system to obtain the desired estimate on the Lipschitz constant of $u$. While studying the equations for the derivatives of $u$ is a natural approach to obtain Lipschitz estimates, we point out the following interesting features of our approach:
\begin{itemize}
    \item From the perspective of PDEs, it might seem more natural to view the equation for each partial derivative $\partial_{x_i}u$ in isolation, rather than viewing the equations together as a system of PDEs for $\nabla u$. In our case, it is essential that we view the equations for $\nabla u$ as a system - if not, we cannot take advantage of the dissipativity condition \eqref{dissipativeintro} for $b$.
    \item Differentiating the PDE for $u$ first and then representing $\nabla u$ in terms of a system of infinite-horizon BSDEs (as we do in the present paper) does not lead to the same system of infinite-horizon BSDEs that we (formally) get from directly differentiating the BSDE in the initial condition of $X$, in the sense of Proposition 4.1 in \cite{karouipengquenez}. In fact, it is not clear that this latter, more probabilistic, approach can be used to obtain the desired estimates. Thus it seems essential that we use a mix of probabilistic and analytical approaches to obtain our main estimate. 
    \item The $\linf$ estimate we obtain for the system of BSDEs representing $\nabla u$ does not follow directly from existing results for infinite horizon BSDEs as in e.g. \cite{BRIAND1998455}, except for in the special caes $d = 1$. Instead, we must combine existing techniques with a judicious application of Girsanov's Theorem, which in turn is only possible because of the special structure of the PDE system for $\nabla u$. 
\end{itemize}
A formal derivation of the estimate is presented in sub-section \ref{subsec:formal}, and a rigorous proof (which requires verifying that $u$ is sufficiently smooth via elliptic regularity theory) is given in sub-section \ref{subsec:rigorous}.

\subsubsection{Application to ergodic control}
In Section \ref{sec:control}, we apply our results to a control problem in which a controller chooses a process $\alpha$ taking values in $\R^d$ in order to minimize the cost 
\begin{align*}
    J(\alpha) = \limsup_{T \to \infty} \frac{1}{T} \E[\int_0^T r(X_s, \alpha_s) ds], 
\end{align*}
subject to $dX_t = \big(b(X_t) + \alpha_t \big) dt + dW_t$.  The inputs to the model are the vector field $b: \R^d \to \R^d$ and the running cost $r : \R^d \times \R^d \to \R$. We focus on the case that the drift is linear in $\alpha$ and the data is sufficiently smooth because this makes it easier to interpret our monotonicity condition. As shown in \cite{Fuhrman2009ErgodicBA}, an optimal control (in the weak formulation) can be characterized by a BSDE with driver $f(x,z) = H(x,z)$, where $H$ is the Hamiltonian given by
\begin{align}
    H(x,z) \coloneqq \inf_{a}h(x,z,a), \text{ where } \nonumber \\
    h(x,z,a) \coloneqq a \cdot z  + r(x,a). 
\end{align}
In this setting, we obtain new results on the existence of optimal controls when the Hamiltonian $H$ has potentially superquadratic growth in $z$.

In Proposition \ref{prop:controlexist}, we show that the driver $H$ satisfies our conditions when $\partial_a^2 r > 0$ and $\partial_a \partial_x r \big(\partial_a^2 r(x,z)\big)^{-1} \geq 0$, in the sense that $z^T \partial_a \partial_x r(x,a) \big(\partial_a^2 r(x,a)\big)^{-1} z \geq 0$ for all $x,a,z \in \R^d$. This is an interesting condition which to the best of the authors' knowledge is new. We first note that this condition is automatically satisfied when $\partial_a \partial_x r = 0$, i.e. when the running cost $r(x,a) = r_0(x) + r_1(a)$ is separable. In this case, we have 
\begin{align*}
    H(x,z) = \tilde{H}(z) + r_0(x), \text{ where }
    \tilde{H}(z) = \inf_{a \in \R^d} \big(a \cdot z + r_1(a)\big). 
\end{align*}
Thus Theorem \ref{thm:main} applies when $\tilde{H}$ is just locally Lipschitz (and in particular can have arbitrary growth) and $r_0$ is bounded and Lipschitz. Thanks to Proposition \ref{prop:ebsdeoptimality}, this implies the existence of optimal controls for this class of ergodic control problems with super-quadratic Hamiltonians. Our result also covers running costs of the form $r(x,a) = \frac{1}{m} |a|^m + r_0(x)$ for $r_0$ Lipschitz and bounded and $1 < m \leq 2$, which is similar to the case considered in \cite{chasseigneichihara} - indeed, this class of examples does not necessarily satisfy $\partial_a^2 r > 0$, but the Hamiltonian is given explicitly by $H(x,z) = \frac{1}{q} |z|^q + r_0(x)$ with $q > 2$ the conjugate of $m$, which clearly satisfies the hypotheses of our main result. 

To interpret the conditions $\partial_a^2 r > 0$ and $\partial_a \partial_x r(x,z)\partial_a^2 g(x,z)^{-1} \geq 0$ when $r$ is not separable, we recall that it is common in studying standard control problems through the stochastic maximum principle to require that the running cost is convex in $a$ and that the Hamiltonian is jointly convex in $x$ and $a$ (see Theorem 4.14 of \cite{Carmona2016LecturesOB}). In this setting, the requirement that $(x,a) \mapsto h(x,p,a)$ is convex is equivalent to the requirement that the $2d \times 2d$ matrix written in blocks as $\begin{pmatrix}
\partial_x^2 r & \partial_x \partial_a r \\
\partial_a \partial_x r & \partial_a^2 r
\end{pmatrix}$
is positive definite, and of course the convexity of $r$ in $a$ means that $\partial_a^2 r > 0$. To summarize, 
\begin{itemize}
    \item when approaching a standard control problem with linear drift through the maximum principle, it is typical to assume $\partial_a^2 r > 0$, $\begin{pmatrix}
\partial_x^2 r & \partial_x \partial_a r \\
\partial_a \partial_x r & \partial_a^2 r
\end{pmatrix} \geq 0$,
\item for ergodic control problems, our main monotonicity condition reduces to the requirement that $\partial_a^2 r > 0$, $\partial_a \partial_x r \big( \partial_a^2 r\big)^{-1} \geq 0$. 
\end{itemize}
Thus while the conditions $\partial_a^2 r >0$, $\partial_a \partial_x r \big( \partial_a^2 r\big)^{-1} \geq 0$ are somewhat difficult to interpret, they are reminiscent of similar conditions which appear naturally in the control literature. For a concrete class of control problems to which this structural condition applies outside the separated case $r(x,a) = r_0(x) + r_1(a)$, consider the running cost
\begin{align*}
r(x,a) = \frac{1}{2}|a|^2 + c_1(x) a + c_0(x). 
\end{align*}
Then 
\begin{align*}
\partial^2_a r(x,a) = I_{d \times d}, \quad \partial_a \partial_x r = Dc_1(x), 
\end{align*}
with $Dc_1$ the Jacobian matric of $c_1 : \R^d \to \R^d$, so our results apply when $c_0$ is Lipschitz, $c_1$ is locally Lipschitz, and $Dc_1(x) \geq 0$. 

We show in sub-section \ref{subsec:risk} that Theorem \ref{thm:main} yields existence of optimal controls for a risk-sensitive version of the ergodic control problem as well. Indeed, in this case the cost functional is of the form 
\begin{align}
    J^{\delta}(\alpha) = \limsup_{T \to \infty} \frac{1}{T} \big( \frac{1}{\delta} \ln \E[\exp(\delta \int_0^T r(X_t^x, \alpha_t)] dt \big)
\end{align}
for some $\delta > 0$,
and the relevant driver is 
\begin{align} \label{riskdriver}
    f(x,z) = H(x,z) + \frac{\delta}{2} |z|^2, 
\end{align}
with $H$ as above. It is easy to see that the driver in \eqref{riskdriver} satisfies the monotonicity condition \eqref{monotoneintro} if and only if the driver $H(x,z)$ does, so we can obtain very similar results in this setting as in the non-risk-sensitive case. Finally, we show in sub-section \eqref{subsec:weakstrong} that because of the Markovian structure of the ergodic control problem, the weak formulation we consider is actually equivalent to an (arguably) more natural strong formulation.

\subsubsection{Application to forward performance processes}

In \cite{Liang2017RepresentationOH}, EBSDEs are used to construct forward performance processes in factor form. In their setting, the diffusion $X$ is the stochastic factor process, and there is a vector $S$ of stock prices satisfying
\begin{align*}
    dS_t^i = S_t^i \big( \eta^i(X_t) dt + \sum_{j = 1}^d \sigma^{ij}(X_t) dW_t^j \big). 
\end{align*}
Thus the dynamics of the stock prices $S^i$ are impacted by the position of the diffusion $X$. We refer to \cite{Liang2017RepresentationOH} for the definition of forward performance processes and their applications. For our purposes, it is enough to note that Proposition 13 of \cite{Liang2017RepresentationOH} reduces the construction of a forward performance process in factor form to solving the EBSDE \eqref{ebsdeintro} with driver 
\begin{align} \label{forwarddriverintro}
    f(x,z) = \frac{1}{2} \delta^2 \text{dist}^2(\Pi, \frac{\theta(x) + z}{\delta}) - z^T \theta(x) + \frac{1}{2} |\theta(x)|^2, 
\end{align}
where $\delta > 0$, $\Pi \subset \R^d$ is closed and convex and $\theta$ is the market price of risk vector given by $\theta(x) = \sigma(x)^T [\sigma(x) \sigma(x)^T]^{-1} \eta(x)$. In the present work, we simply take the market price of risk vector $\theta : \R^d \to \R^d$ as a given, and ask - what conditions on $\theta$ and $\Pi$ guarantee that the driver $f$ appearing in \eqref{forwarddriverintro} satisfies our conditions? By explicitly computing $\partial_x f$ for $f$ given by \eqref{forwarddriver}, we are able to state in Proposition \ref{prop:forward} a general existence criterion which can be checked on a case-by-case basis.

\subsection{Organization of the paper}

In the remainder of the introduction, we fix notations and conventions. In Section \ref{sec:prelim} we discuss some preliminaries. In Section \ref{sec:existence} we state our main existence result, Theorem \ref{thm:main}. In sub-section \ref{subsec:formal}, we provide a formal derivation of the main a-priori estimate, and sub-section \ref{subsec:rigorous} is devoted to a rigorous proof of Theorem \ref{thm:main}. In Sections \ref{sec:control} and \ref{sec:forward} we present applications to ergodic control and forward performance processes, respectively. Finally, the Appendix collects some auxiliary results concerning the relevant PDEs and their connection to BSDEs. 

\subsection{Notation and conventions}

\paragraph{Probabilistic setup} We work on a fixed probablility space $(\Omega, \sF, \pr)$, which hosts a $d$-dimensional Brownian motion $W$. We denote by $\bF = (\sF_t)_{t \geq 0}$ the augmented filtration of the Brownian motion $W$. We denote by $\E[\cdot]$ expectation with respect to the probability measure $\pr$. 

\paragraph{Spaces of processes} We denote by $\sP$ the space of progressively measurable processes taking values in some Euclidean space. When necessary, we specify the target in a natural way, e.g. $\sP(\R^n)$ denotes the space of progressively measurable processes taking values in $\R^n$. $\sP^2$ is the space of $Z \in \sP$ such that 
\begin{align*}
    \norm{Z}_{\sP^2}^2 \coloneqq \E[\int_0^{\infty} |Z_t|^2 dt] < \infty
\end{align*}
For $1 \leq p \leq \infty$, we denote by $\lpee$ the usual space of $p$-integrable random variables, and by $\spee$ the space of continuous, adapted processes $X$ such that \begin{align*}
    \norm{X}_{\spee} \coloneqq \norm{\sup_{0 \leq t \leq \infty} |X_t|}_{\lpee} < \infty. 
\end{align*}We also define and notate localized versions of these spaces in the usual way, e.g. $\sP^2_{loc}$ is the space of locally square integrable processes. Finally, we note that we use the standard notation $\mathcal{E}(\cdot)$ for the stochastic exponential, i.e. if $M$ is a martingale then $\mathcal{E}(M)$ denotes the local martingale $\mathcal{E}(M)_t = \exp(M_t - \frac{1}{2} \langle M \rangle_t)$, $\langle M \rangle$ being the quadratic variation of $M$.

\paragraph{Spaces of functions}
We say that a function $u$ defined on some open set in $\R^d$ is $C^k$ if it is $k$-times continuously differentiable. 
Given an open set $\Omega \subset \R^d$, we define
\begin{align*}
    [u]_{\calpha(\Omega)} = \sup_{x,y \in \Omega, x \neq y} \frac{|u(x) - u(y)|}{|x-y|^{\alpha}}. 
\end{align*}
We denote by $C^{\alpha}(\Omega)$ the space of all functions $u$ defined on $\Omega$ and taking values in some Euclidean space such that 
\begin{align*}
    \norm{u}_{\calpha(\Omega)} \coloneqq \norm{u}_{\linf(\Omega)} + [u]_{\calpha(\Omega)} < \infty
\end{align*}
We denote by $C^{k,\alpha}(\Omega)$ the space of all functions $u$ taking values in some Euclidean space such that 
\begin{align*}
    \norm{u}_{C^{k,\alpha}(\Omega)} \coloneqq \sup_{|\beta| < k} \norm{D^{\beta}}_{\linf(\Omega)} + \sum_{ |\beta| = k} \norm{D^{\beta} u}_{\calpha(\Omega)} < \infty. 
\end{align*}
Here $|\beta|$ denotes the order of a multi-index $\beta$. We use the shorthand $C^{k,\alpha} = C^{k,\alpha}(\R^d)$. Finally, we define and notate localized versions of all these spaces, e.g. $u \in \calpha_{loc}(\Omega)$ means that for all open sets $U$ compactly contained in $\Omega$, $u \in \calpha(U)$. On one occasion, we will have need of the Sobolev spaces $H_0^1(\Omega), W^{2,p}(\Omega)$, etc., and we refer to \cite{Gilbarg1977EllipticPD} for the relevant definitions. Finally, we note that throughout the paper we will use $B_R$ to denote the unit ball in $\R^d$. 

\paragraph{Notation for stochastic integrals}

When $Z \in \sP^2_{loc}(\R^d)$, we will write $\int Z \cdot dW$ for the local martingale $\sum_{j = 1}^d \int Z^j dW^j$, where $Z^j$ is the $j^{th}$ component of $Z$. At times, we will also deal with integrands $Z \in \sP^2_{loc}(\R^{d \times d})$. In this case, we will write $Z^i$ for the $i^{th}$ row of $Z$, and then $\int Z dW$ will denote the $\R^d$-valued local martingale whose $i^{th}$ component is given by $\big(\int Z dW\big)^i = \int Z^i dW = \sum_{ij} \int Z^{ij} dW^j$, where now $Z^{ij}$ denotes the element of $Z$ in row $i$ and column $j$ (this is just the usual way of interpreting $Z dW$ though matrix multiplication). The use of these two different conventions (rather than viewing all integrands as taking values in $\R^{n \times d}$ for some $n$ and always interpreting $Z dW$ by matrix multiplication, for example) is motivated by the fact that we will usually have $Z = \nabla u(X)$ for some $ u: \R^d \to \R$ (in which case the first convention is natural) or $z = Dv(X)$ for some $v : \R^d \to \R^d$ (in which case the second convention is more natural). 

\paragraph{Notation for derivatives}
We use the variables $x = (x_1,...,x_d)$ and $z = (z_1,...,z_d)$ for elements of $\R^d$. The notation $\partial_{x_i}$ means differentiation in the variable $x_i$, and similarly for $\partial_{z_j}$. For $b = b(x) : \R^d \to \R^d$, we use $Db$ to denote the matrix $(Db)_{ij} = \partial_j b^i$. For a map $f = f(x,z) : \R^d \times \R^d \to \R^d$, we write $\partial_x^2 f$ for the Hessian of $f$ in the variable $x$, i.e. $(\partial_x^2 f(x,z))_{ij} = \partial_{x_i} \partial_{x_j} f(x,z)$. Likewise, $\partial_z^2 f(x,z)$ is the Hessian in the variable $z$, while $\partial_z \partial_x f$ denotes the matrix $\big(\partial_z \partial_x f(x,z)\big)_{ij} = \partial_{z_j} \partial_{x_i} f(x,z)$. The notation $\partial_{x} \partial_z f$ is similar. Thus $\partial_z \partial_x f$ is the derivative of the map $z \mapsto \partial_x f(x,z)$, and of course $\partial_z \partial_x f = \big(\partial_x \partial_z f\big)^T$. 

\section{The main result} \label{sec:prelim}
We recall the ergodic BSDE 
\begin{align} \label{ebsde}
    \begin{cases}
    dX_t = b(X_t) dt +  dW_t, \\
    dY_t = - f(X_t,Z_t) dt + \lambda dt + Z_t \cdot dW_t.
    \end{cases}
\end{align}
The data of the problem consists of a vector field $b = b(x) : \R^d \to \R^d$ and a driver $f = f(x,z) : \R^d \times \R^d \to \R$. The diffusion $X$ takes values in $\R^d$, and the unknowns $Y$ and $Z$ take values in $\R$ and $\R^d$, respectively. We are looking for a \textbf{Markovian solution} to \eqref{ebsde}, i.e. a triple $(u,v,\lambda)$ with $\lambda \in \R$ and $(u,v) : \R^d \to \R \times \R^d$ such that for any $x \in \R^d$, the processes $Y^x = u(\cdot, X^x)$, $Z^x = v(\cdot, X^x)$ satisfy 
\begin{align*}
    Y_t^x = Y_T^x + \int_t^T f(X^x_s,Z^x_s) ds - \lambda(T-t) - \int_t^T Z^x_s \cdot dW_s
\end{align*}
for all $0 \leq t \leq T < \infty$. Here $X^x$ denotes the solution of the SDE 
\begin{align*}
    X_t^x = x + \int_0^t b(X_s^x) ds + W_t. 
\end{align*}
We recall also the corresponding PDE 
\begin{align} \label{pde}
    \frac{1}{2} \Delta u + b \cdot \nabla u + f(x,\nabla u) = \lambda \text{ in } \R^d. 
\end{align}
We will make the following assumptions regarding the data $b$ and $f$. 
\begin{assumption} \label{assump:existence}
The vector field $b : \R^d \to \R^d$ and the driver $f = g+h: \R^d \times \R^d \to \R$ satisfy
\begin{enumerate}
    \item (dissipativity and Lipschitz property for $b$) The vector field $b$ is Lipschitz, and there is a constant $\delta > 0$ such that \begin{align} \label{dissipative} 
        \big(b(x) - b(\bx)\big) \cdot (x - \bx) \leq - \delta |x - \bx|^2
    \end{align}
    holds for all $x, \bar{x} \in \R^d$. 
    \item (growth and regularity for $g$) The function $g = g(x,z)$ is Lipschitz on $\R^d \times B_R$ for each $R > 0$. Moreover $\sup_{x \in \R^d} |g(x,0)| < \infty$. 
     \item (monotonicity of $\partial_x g$) The estimate
    \begin{align} \label{monotone}
   \big(\partial_x g(x,z) - \partial_x g(x, \bz)\big) \cdot (z - \bz) \leq 0
\end{align}
holds for almost every $(x,z,\bz) \in (\R^d)^3$, where $\partial_x g$ is the weak derivative of $g$ in the variable $x$.  
    \item (growth and regularity for $h$) The function $h = h(x,z) : \R^d \times \R^d \to \R$ is globally Lipschitz in $x$, i.e. $|h(x,z) - h(\bx,z)| \leq L|x - \bx|$ for all $x,\bx, z \in \R^d$ and some $L > 0$. Moreover, $h$ is Lipschitz (in both arguments) on $\R^d \times B_R$ for each $R > 0$, and $\sup_{x \in \R^d} |h(x,0)| < \infty$. 
\end{enumerate}
\end{assumption}
\begin{remark}
The first condition \eqref{dissipative} is standard in the ergodic BSDE literature, and guarantees that the diffusion $X$ is exponentially ergodic, in the sense that 
\begin{align*}
    |X^x_t - X^{\bx}_t| \leq \exp(-2\delta t)|x - \bx|^2
\end{align*}
a.s. for each $t \geq 0$. 
\newline \newline 
The growth and regularity conditions on $g$ and $h$ ensure that if $\pi$ is an appropriate truncation function, then $\tilde{f}(x,z) = f(x,\pi(z))$ satisfies standard Lipschitz conditions for ergodic BSDEs, which will be important for the truncation technique we use to prove Theorem \ref{thm:main}. Setting $g = 0$, we see that we easily cover the case that $f$ is globally Lipschitz and $f(\cdot,0)$ is bounded, which is the condition used in \cite{Fuhrman2009ErgodicBA} and \cite{Debussche2010ErgodicBU}, albeit in a more general setting. 
\newline \newline 
The third condition concerns the monotonicity of the map $z \mapsto  \partial_xg(x,z)$, and is our main qualitative assumption. We note that when  $d = 1$ and $g$ is $C^2$, the condition \eqref{monotone} is equivalent to the simpler condition $\partial_{zx} g \leq 0$. 
\end{remark}

We are now ready to state the main result. 
We define the quantity $M$ by
\begin{align} \label{mdef}
    M \coloneqq \norm{ \partial_xg}_{\linf(\R^d \times B_1)} + \norm{\partial_x h}_{\linf(\R^d \times \R^d)}
\end{align}
The main result is the following:
\begin{thm} \label{thm:main}
Suppose that Assumption \ref{assump:existence} holds. Then, there exists a Markovian solution $(u, v, \lambda)$ to \eqref{ebsde}. The function $u$ has linear growth, and $v$ satisfies
\begin{align*}
    \norm{v}_{\linf} \leq \frac{M}{\delta}. 
\end{align*}
Moreover, the solution is unique in the class of Markovian solutions with $v$ bounded and $u$ of linear growth, in the sense that if $(\bar{u},\bar{v},\bar{\lambda})$ is another solution with the same properties then $v = \bar{v}$, $\lambda = \bar{\lambda}$, while $u$ and $\bar{u}$ differ by a constant. Finally, $u$ is in fact a classical solution of the PDE \eqref{pde}, and $v = \nabla u$ a.e.
\end{thm}
\begin{remark} \label{rmk:uniqueness}
The uniqueness part of Theorem \ref{thm:main} follows easily from uniqueness for standard Lipschitz ergodic BSDEs, which can be obtained by adapting the proof of Theorem 3.11 in \cite{Debussche2010ErgodicBU}. More precisely, if $(u,v,\lambda)$ and $(\bar{u},\bar{v}, \bar{\lambda})$ are Markovian solutions to \eqref{ebsde}, then they are also both solutions to the ergodic BSDE 
\begin{align} \label{ebsdetrunc1}
    dY_t = \lambda dt - \tilde{f}(X_t,Z_t) dt + Z_t \cdot dW_t,
\end{align}
where $\tilde{f}(X_t, Z_t) = \tilde{f}(X_t,\pi(Z_t))$ and $\pi : \R^d \to \R^d$ is the orthogonal projection onto the ball of radius $\max \{\norm{v}_{\linf}, \norm{\bar{v}}_{\linf}\}$. Furthermore $\tilde{f}$ is globally Lipschitz thanks to Assumption \ref{assump:existence}. We can thus conclude from uniqueness for the equation \eqref{ebsdetrunc1} that $v = \bar{v}$, $\lambda = \bar{\lambda}$, and $u$ and $\bar{u}$ differ by a constant, as required. Thus, in what follows we focus on the question of existence. 
\end{remark}

\section{Existence for EBSDEs under monotonicity} \label{sec:existence}

The goal of this section is to present a proof of Theorem \ref{thm:main}. 
\newline \newline 
For later use, we recall the following elementary fact, which gives a differential characterization of the dissipativity condition \eqref{dissipative}. 

\begin{prop} \label{prop:dissipative}
Suppose that $\phi : \R^d \to \R^d$ is $C^1$. Then $\phi$ satisfies the dissipativity condition
\begin{align} \label{dissipativelemma}
    \big(\phi(x) - \phi(\bx)\big) \cdot (x - \bx) \leq - \epsilon |x - \bx|^2
\end{align}
holds for $\epsilon \geq 0$ if and only if the condition
\begin{align*}
    z^T D\phi(x) z \leq - \epsilon |z|^2
\end{align*}
holds for all $x,z \in \R^d$. 
\end{prop}


Our approach to proving Theorem \ref{thm:main} will be by truncation. More precisely, first choose a smooth increasing function $\rho : [0,\infty) \to [0,\frac{M}{\delta} + 1]$ such that 
\begin{itemize}
    \item $\rho(x) = x$ for $x \leq \frac{M}{\delta}$, 
    \item $\rho(x) = \frac{M}{\delta} + 1$ for all $x > \frac{M}{\delta} + 2$. 
\end{itemize}
Then we set 
\begin{align} \label{pidef}
    \pi(z) = \frac{\rho(|z|)}{|z|} z,
\end{align}
and note that $\pi$ is smooth, radial, bounded, and $\pi(z) = z$ when $|z| \leq \frac{M}{\delta}$. Next, we set 
\begin{align} \label{tildef}
    \tilde{f}(x,z) = f(x,\pi(z)), 
\end{align}
and note that $\tilde{f}$ is Lipschitz, so the ergodic BSDE 
\begin{align} \label{ebsdetrunc}
    dY_t = \lambda dt - \tilde{f}(X_t,Z_t) dt + Z_t \cdot dW_t
\end{align}
has a Markovian solution $(u,v,\lambda)$ with $v$ bounded. If we can show that $\norm{v}_{\linf} \leq \frac{M}{\delta}$, then $(u,v,\lambda)$ is in fact a Markovian solution \eqref{ebsde}. Actually, the proof only works as outlined here when $b$ and $f$ are sufficiently smooth. The general case is more subtle, but the basic strategy is the same. 

\subsection{Formal derivation of the a-priori estimate} \label{subsec:formal}

As explained above, proving Theorem \ref{thm:main} boils down to establishing the estimate 
\begin{align} \label{mainest}
    \norm{v}_{\linf} \leq \frac{M}{\delta}
\end{align}
whenever $(u,v,\lambda)$ is a Markovian solution to \eqref{ebsde} with $f$ globally Lipschitz and satisfying Assumption \ref{assump:existence}. For the reader's convenience, we give in this sub-section a formal derivation of \eqref{mainest} which provides intuition while avoiding the technical details. 
\newline \newline
We suppose in our formal argument that $d = 1$ (this simplifies a key step in the derivation), and we also suppose that $u$ is in fact a smooth solution of the PDE 
\begin{align} \label{onedpde}
   \frac{1}{2} \partial_x^2 u + b \partial_x u + f(x,\partial_x u) = \lambda, 
\end{align}
and that $v = \partial_x u$. Then we differentiate \eqref{onedpde} to find that 
\begin{align*}
    \frac{1}{2} \partial_x^2 v + b' v  + \partial_x f(x,v) + \big(\partial_z f(x, v) + b\big) \partial_x v = 0.
\end{align*}
We see that in fact $(\tilde{Y}, \tilde{Z}) = (v(X), \partial_x v(X))$ solves the infinite horizon BSDE 
\begin{align*}
    d\tilde{Y}_t = - \psi(X_t, \tilde{Y}_t, \tilde{Z}_t) dt + \tilde{Z}_t \cdot dW_t, 
\end{align*}
with \begin{align*}
    \psi(x,y,z)
    = b'(x) y + \partial_x g(x,y) + \partial_x h(x,v(x)) + \partial_z f(x,v(x)) z
\end{align*}
The dissipativity condition \eqref{dissipative} and the monotonicity condition \eqref{monotone} combine to show that $\psi$ is monotone in $y$, and $(\psi(x,y,z) - \psi(x,\bar{y},z))(y - \bar{y}) \leq - \delta |y - \bar{y}|^2$, so that estimates for infinite horizon BSDEs (see the proof of Lemma 3.1 in \cite{BRIAND1998455}) show that 
\begin{align*}
   \norm{v}_{\linf} = \norm{\tilde{Y}}_{\sinf} \leq \frac{1}{\delta} \sup_{x \in \R^d} |\psi(x,0,0)| \leq \frac{M}{\delta},
\end{align*}
as desired.

\subsection{Proof of Theorem \ref{thm:main}} \label{subsec:rigorous}

This sub-section is devoted to a proof of Theorem \ref{thm:main}. We use heavily the connection between the ergodic BSDE \eqref{ebsde} and the PDE \eqref{pde} established in the appendix. 

\begin{lem}[a-priori estimate] \label{lem:apriori}
Suppose Assumption \ref{assump:existence} holds, and in addition $f = g+ h$ is globally Lipschitz and $b$,$g$,$h$ are $C^{1,\alpha}_{loc}$. Let $(u,v)$ be the unique Markovian solution to the ergodic BSDE \eqref{ebsde} with $u$ of linear growth and $v$ bounded. Then we have
\begin{align*}
    \norm{v}_{L^{\infty}} \leq \frac{M'}{\delta}, 
\end{align*}
where $M'= \sup_{x \in \R^d} |\partial_x g(x,0)| + \norm{\partial_x h}_{\linf(\R^d \times \R^d)}$. 
\end{lem}
\begin{proof}
By Lemma \ref{prop:regularityebsde}, we find that $u$ is a $C^3$ solution to \eqref{pde}, and that we may assume $v = \nabla u$ is $C^2$. Thus, we can differentiate \eqref{pde} to find an equation for $v$. Namely, writing $v^{i} = \partial_{i} u$, we obtain the PDE system 
\begin{align} \label{gradeqn}
    \frac{1}{2} \Delta v^{i} + \partial_i b \cdot v + b \cdot \nabla v^{i} +  \partial_{x_i}f(x,v) + \partial_z f(x,v) \nabla v^{i} = 0. 
\end{align}
We notice that if we set $X^{x}$ to be the unique solution of the SDE
\begin{align*}
    X_t^x = x + \int_0^t b(X_s^x) ds + W_t,
\end{align*}
then since $v$ is $C^2$, we may apply It\^o's formula and \eqref{gradeqn} to verify that the $\R^d \times \R^{d \times d}$-valued pair of processes $(Y^x,Z^x)$ given by $(Y^x, Z^x) = (v(X^x), Dv(X^x))$ solves the discounted BSDE 
\begin{align*}
    Y_t^{x,i} = Y_T^{x,i} + \int_t^T \big(\partial_i b(X^x_s) \cdot Y^x_s + \partial_{x_i} f(X^x_s,Y^x_s)\big) ds \\ + \int_t^T \partial_z f(X^x_s,Y^x_s) \cdot Z^{x,i}_s ds - \int_t^T Z_s^{x,i} \cdot dW_s. 
\end{align*}
We rewrite the above as
\begin{align} \label{diffbsde}
     Y_t^{x,i} = Y_T^{x,i} + \int_t^T \big(\partial_i b(X^x_s) \cdot Y^x_s + \partial_{x_i} f(X^x_s,Y^x_s)\big) ds - \int_t^T Z_s^{x,i} \cdot d\tilde{W}_s \nonumber \\
     = Y_T^{x,i} + \int_t^T \big(\partial_i b(X^x_s) \cdot Y^x_s  + \partial_{x_i} g(X^x_s,Y^x_s) + \partial_{x_i} h(X^x_s,Y^x_s)\big) ds - \int_t^T Z_s^{x,i} \cdot d\tilde{W}_s
\end{align}
where 
\begin{align*}
\tilde{W} \coloneqq W - \int \partial_z f(X^x,Y^x) dt
\end{align*}
is by Girsanov's Theorem a Brownian motion on $[0,T]$ under the equivalent probability measure $\Q_T$ given by 
\begin{align*}
    d \Q_T = \mathcal{E}(\int \partial_z f(X^x,Y^x) \cdot dW)_T d\pr. 
\end{align*}
Now we apply It\^o's formula together with \eqref{diffbsde} to compute
\begin{align} \label{y2dynamics}
    d|Y_t^x|^2 = \Big(-2(Y_t^x)^T  Db(X^x_t) Y_t - 2(Y^x_t)^T \partial_x g(X_t^x,Y_t^x) \nonumber \\ - 2(Y^x_t)^T \partial_x h(X_t^x,Y_t^x) + |Z_t^x|^2 \Big) dt + (Y_t^{x})^T Z_t^x  d\tilde{W}_t \nonumber \\
    = \Big(-2(Y_t^x)^T  Db(X^x_t) Y_t - 2(Y^x_t)^T \big( \partial_x g(X_t^x,Y_t^x) - \partial_x g(X_t^x,0)\big) \nonumber \\  - 2(Y^x_t)^T \partial_x g(X_t^x,0) -2 (Y^x_t)^T \partial_x h(X_t^x,Y_t^x)+ |Z_t^x|^2 \Big) dt + (Y_t^{x})^T Z_t^x  d\tilde{W}_t.
\end{align}
The monotonicity condition gives 
\begin{align} \label{y2monotone}
    \beta_t \coloneqq -2(Y_t^x)^T  Db(X^x_t) Y_t - 2(Y^x_t)^T \big( \partial_x g(X_t^x,Y_t^x) - \partial_x g(X_t^x, 0)\big) + |Z_t^x|^2 \geq 2\delta |Y_t^x|^2, 
\end{align}
and Young's inequality together with the growth conditions on $\partial_{x} g$ and $\partial_x h$ gives 
\begin{align} \label{y2growth}
    |\gamma_t| \coloneqq |2(Y^x_t)^T \partial_x g(X_t^x,0) + 2(Y^x_t)^T \partial_x h(X_t^x,Y_t^x)| \nonumber \\ \leq 2|Y^x_t| \big(|\partial_x g(X^x_t,0)| + |\partial_x  h(X^x_t,Y^x_t)|\big) \nonumber  \\
    \leq \frac{1}{\delta} \big(|\partial_x g(X^x_t,0)| + |\partial_x  h(X^x_t,Y^x_t)|\big)^2 + \delta|Y^x_t|^2 
    \leq \frac{(M')^2}{\delta} + \delta |Y^x_t|^2. 
\end{align}
Combining \eqref{y2dynamics}, \eqref{y2monotone}, and \eqref{y2growth}, we see that we have 
\begin{align} \label{y2eqn}
    d|Y_t^x|^2 = \big(\delta |Y_t^x|^2 + \alpha_t\big) dt + (Y_t^x)^T Z_t^x d\tilde{W}_t,
\end{align}
where 
\begin{align*}
    \alpha_t = \beta_t + \gamma_t - \delta |Y_t^x|^2 \geq - \frac{(M')^2}{\delta}. 
\end{align*}
Now from \eqref{y2eqn} it is easy to check that 
\begin{align} \label{almostdone}
    |Y_0^x|^2 = \exp(-\delta T) |Y_T^x|^2 - \int_0^T \exp(- \delta s) \alpha_s ds - \int_0^T \exp(-\delta s) (Y_s^x)^T Z_s^x d\tilde{W}_t. 
\end{align}
We already known that $Y^{x}$ is bounded, so from equation \eqref{diffbsde} we see that the $\Q_T$-local martingale $\int Z^x \tilde{W}$ is bounded. It follows that the stochastic integral in \eqref{almostdone} is a true martingale on $[0,T]$ under $\Q_T$, and so
\begin{align*}
     |Y_0^x|^2 = \exp(-\delta T) \E^{\Q_T}[|Y_T^x|^2] + \E^{\Q_T}[\int_0^T -\exp(- \delta s) \alpha_s ds] \\
    \leq \exp(-\delta T) \norm{Y}_{\sinf}^2 + \int_0^T \exp(-\delta s) \frac{(M')^2}{\delta} dt \\
    = \exp(-\delta T) \norm{Y}_{\sinf}^2 + \frac{(M')^2}{\delta^2}\big(1 - \exp(-\delta T)), 
\end{align*}
where $\E^{\Q_T}$ denotes expectation under $\Q_T$. Finally, sending $T \to \infty$ gives the estimate 
\begin{align*}
    |v(x)|^2 = |Y_0^x|^2 \leq \frac{(M')^2}{\delta^2}, 
\end{align*}
i.e. $\norm{v}_{L^{\infty}} \leq \frac{M'}{\delta}$. 
\end{proof}

\begin{remark} \label{rmk:weakmonotone}
A careful reading of the proof of Lemma \ref{lem:apriori} shows that we do not in fact need the full monotonicity condition \eqref{monotone} for the conclusion of the lemma to hold, but only the weaker condition 
\begin{align} \label{weakmonotone}
    \big(\partial_x g(x,z) - \partial_x g(x,0)\big) \cdot z \leq 0
\end{align}
for all $x, z \in \R^d$. 
\end{remark}
Finally, we can complete the proof of Theorem \ref{thm:main}.

\begin{proof}[Proof of Theorem \ref{thm:main}] 
We split the argument into several steps. 
\newline \newline
\textit{Step 1 (checking that mollification preserves structural properties)}: Let $(\rho_{\epsilon})_{0 <\epsilon < 1}$ be a standard approximation to the identity on $\R^d$ (in particular satisfying $\supp(\rho_{\epsilon}) \subset B_1$ for all $\epsilon)$, and define
\begin{itemize}
    \item $b_{\epsilon}(x) = b * \rho_{\epsilon}(x) = \int_{\R^d} b(x - y) \rho_{\epsilon}(y) dy$
    \item $g_{\epsilon}(x,z) = \int_{\R^d} \int_{\R^d} g(x - y_1, z - y_2) \rho_{\epsilon}(y_1) \rho_{\epsilon}(y_2) dy_1 dy_2$, 
    \item $h_{\epsilon}(x,z) = \int_{\R^d} \int_{\R^d} h(x - y_1, z - y_2) \rho_{\epsilon}(y_1) \rho_{\epsilon}(y_2) dy_1 dy_2$. 
    \item $f_{\epsilon} = g_{\epsilon} + h_{\epsilon}$. 
\end{itemize}
Then $b_{\epsilon}$ and $f_{\epsilon}$ are smooth and, moreover we have 
\begin{align} \label{check1}
   \big(b_{\epsilon} (x) - b_{\epsilon}(\bx)\big) \cdot (x - \bx) = \big(\int_{\R^d} \big[ b(x-y) - b(\bx - y) \big]\rho_{\epsilon}(y) dy\big) \cdot (x - \bx) \nonumber \\ = \int_{\R^d} \big( \big[b(x-y) - b(\bx - y)\big] \rho_{\epsilon}(y)\big) \cdot (x - y - \bx + y) dy \leq - \delta |x - \bx|^2,
\end{align}
i.e. $b_{\epsilon}$ satisfies the dissipativity condition \eqref{dissipative} uniformly in $\epsilon$. Similar computations show that  
\begin{align} \label{check2}
    \big(\partial_x g_{\epsilon}(x,z) - \partial_x g_{\epsilon}(x,\bz)\big) \cdot (z - \bz) \leq 0,
\end{align}
and also
\begin{align*}
    \sup_{x \in \R^d} |\partial_x g_{\epsilon}(x,0)| \leq \norm{\partial_x g}_{\linf(\R^d \times B_1)}, \nonumber \\
    \norm{\partial_x h_{\epsilon}}_{\linf(\R^d \times \R^d)} \leq \norm{\partial_x h}_{\linf(\R^d \times \R^d}),
\end{align*}
so that in particular 
\begin{align} \label{check3}
    \sup_{x \in \R^d} |\partial_x g_{\epsilon}(x,0)| +
    \norm{\partial_x h_{\epsilon}}_{\linf(\R^d \times \R^d)} \leq M. 
\end{align}
\textit{Step 2 (estimates for mollified and truncated equation)}:
We set $\tilde{g}_{\epsilon}(x,z) = g_{\epsilon}(x,\pi(z))$, $\tilde{h}_{\epsilon}(x,z) = h_{\epsilon}(x,\pi(z))$, and $\tilde{f}_{\epsilon} = \tilde{g}_{\epsilon} + \tilde{h}_{\epsilon}$,
where $\pi$ is as introduced in \eqref{pidef}
and notice that thanks to \eqref{check3}
\begin{align} \label{check4}
    \big(\partial_x \tilde{g}_{\epsilon}(x,z) - \partial_x \tilde{g}_{\epsilon}(x,0)\big) \cdot z  = \rho(|z|)\big(\partial_x \rho(|z|) g_{\epsilon}(x,\pi(z)) - \partial_x \tilde{g}_{\epsilon}(x,0)\big) \cdot \pi(z) \leq 0. 
\end{align}
That is, $\tilde{g}_{\epsilon}$ satisfies the weak monotonicity condition \eqref{weakmonotone}. 
Since $g_{\epsilon}$ and $h_{\epsilon}$ are smooth, so are $\tilde{g}_{\epsilon}$ and $\tilde{h}_{\epsilon}$, and from \eqref{check3} we see that 
\begin{align} \label{check5}
    \sup_{x \in \R^d} |\partial_x \tilde{g}_{\epsilon}(x,0)| + \norm{\partial_x \tilde{h}_{\epsilon}}_{\linf(\R^d \times \R^d)} \leq M. 
\end{align}
Using \eqref{check1}, \eqref{check4}, and \eqref{check5}, we may apply Lemma \ref{lem:apriori} (actually a slight extension, see Remark \ref{rmk:weakmonotone}) to conclude that 
\begin{align} \label{convest}
    \norm{v^{\epsilon}}_{\linf} \leq \frac{M}{\delta}, 
\end{align}
where $(u^{\epsilon}, v^{\epsilon},\lambda^{\epsilon})$ is the unique Markovian solution to \begin{align} \label{approxebsde}
    dY_t^{\epsilon} = \lambda^{\epsilon} dt - \tilde{f}_{\epsilon}(X_t,Z^{\epsilon}_t) dt + Z_t^{\epsilon} \cdot dW_t. 
\end{align}
\textit{Step 3 (Passing the a-priori estimate through the limit)}: 
Recalling that $u^{\epsilon}$ is unique only up to an additive constant, let us normalize $u^{\epsilon}$ so that $u^{\epsilon}(0) = 0$. By Proposition \ref{prop:regularityebsde},  $u^{\epsilon}$ actually is a $C^2$ solution of the the corresponding PDE and $v^{\epsilon} = \nabla u^{\epsilon}$ a.e., thus elliptic regularity arguments (very similar to those presented in the proof of Proposition \ref{prop:regularityebsde} in the appendix) show that for some $0 < \gamma < 1$ and for each $R > 0$, $u^{\epsilon}$ is uniformly bounded in $C^{2,\gamma}(B_R)$, and that moreover for some $\epsilon_n \downarrow 0$, we have
\begin{align*}
    \norm{u^{\epsilon_n} - u}_{C^{2,\gamma}(B_R)} \to 0 \text{ as } n\to \infty, 
\end{align*}
for each $R > 0$, where $(u,v,\lambda)$ is the unique Markovian solution of the truncated EBSDE \eqref{ebsdetrunc} (also normalized so that $u(0) = 0$). In particular, $v^{\epsilon} = \nabla u^{\epsilon_n} \to \nabla u = v$ almost everywhere, 
and so we get $\norm{v}_{\linf} \leq \frac{M}{\delta}$. Thus, the unique Markovian solution $(u,v,\lambda)$ of \eqref{ebsdetrunc} is also a Markovian solution of \eqref{ebsde}, as desired. As explained in Remark \ref{rmk:uniqueness}, uniqueness follows from uniqueness for Lipschitz drivers, and the fact that $u$ is a regular solution of \eqref{pde} is established in Proposition \ref{prop:regularityebsde}. 
\end{proof}

\section{Applications to control} \label{sec:control}
In this section, we apply Theorem \ref{thm:main} to an ergodic control problem. We do not pursue the greatest possible generality, because we are primarily interested in answering the question - what does the monotonicity condition \eqref{monotone} mean in terms of ergodic control? In particular, we focus in both cases on the case of linear drift, convex costs, and sufficiently smooth data, since this is the setting where the monotonicity condition  \eqref{monotone} can most easily be understood. 

\subsection{Ergodic control} \label{subsec:ergodiccontrol}

In this sub-section, we recall the basic connection between ergodic control and the EBSDE \eqref{ebsde}, and then we check which conditions on the data of the ergodic control problem guarantee that the corresponding EBSDE satisfies Assumption \ref{assump:existence}. 

\subsubsection{Set-up and connection to ergodic BSDE}

The data will be a cost function $r = r(x,a) : \R^d \times \R^d \to \R$ and a vector field $b: \R^d \to \R^d$. 
\begin{assumption}[basic assumptions for control problem] \label{assump:ergcontrol}
The functions $b : \R^d \to \R^d$ and $r : \R^d \times \R^d \to \R$ satisfy 
\begin{enumerate}
    \item (dissipativity and regularity for the drift) $b$ satisfies the first condition in Assumption \ref{assump:existence}, i.e. $b$ is Lipschitz and satisfies the dissipativity condition \eqref{dissipative}. 
    \item (local boundedness of the cost) $r : \R^d \times \R^d \to \R$ satisfies the estimate
    \begin{align} \label{quadgrowth}
        |r(x,a)| \leq \kappa(|a|)
    \end{align}
    for some increasing function $\kappa : \R_+ \to \R_+$ and all $x,a \in \R^d$.
    \item (convexity of the cost) $r$ is uniformly convex in $a$.
\end{enumerate}
\end{assumption}
For each $x \in \R^d$, we denote by $X^x$ the solution of the SDE \begin{align*}
    X_t^x = x + \int_0^t b(X_s^x) ds + W_t. 
\end{align*}
We define $\sA$ to be the set of bounded and progressively measurable processes taking values in $\R^d$. 
For any $\alpha \in \sA$, we can write
\begin{align*}
    X_t^x = \big(b(X_t^x) + \alpha_t\big) dt + dW_t^{\alpha},
\end{align*}
where $W^{\alpha} = W- \int \alpha dt$ is a Brownian motion on the interval $[0,T]$ under the probability measure $\pr^{\alpha,T}$, given by 
\begin{align*}
    d\pr^{\alpha,T} = \mathcal{E}(\int \alpha  \cdot dW)_T d\pr. 
\end{align*}
Now we define for each $x \in \R^d$ the cost functional $J_x : \sA \to \R$, given by
\begin{align} \label{costfunctional}
    J_x(\alpha) = \limsup_{T \to \infty} \frac{1}{T} \E^{\pr^{\alpha,T}}[\int_0^T r(X_t^x,\alpha_t) dt], 
\end{align}
which is well-defined under Assumption \ref{assump:ergcontrol}.
We recall the relevant Hamiltonian 
\begin{align} \label{hamiltonian}
    H(x,z) = \inf_{a \in \R^d} h(x,z,a), \nonumber \\
    h(x,z,a) \coloneqq a \cdot z + r(x,a). 
\end{align}
Since $r$ is uniformly convex, there is an optimizer $\hat{a}$ of $h$, i.e. a measurable map $\hat{a}(x,z): \R^d \times \R^d \to \R^d$ such that 
\begin{align*}
    h(x,z, \hat{a}(x,z)) = H(x,z). 
\end{align*}
We pose the ergodic BSDE 
\begin{align} \label{ebsdecontrol}
    dY_t = \lambda dt - H(X_t,Z_t) dt + Z_t \cdot dW_t.
\end{align}
We now state without proof the connection between the control problem and an ergodic BSDE, which is just a slight adaptation of Theorem 7.1 in \cite{Fuhrman2009ErgodicBA} to our setting. 
\begin{prop} \label{prop:ebsdeoptimality}
Suppose that the ergodic BSDE \eqref{ebsdecontrol} has a Markovian solution $(u,v,\lambda)$ such that $u$ is of linear growth and $v$ is bounded, and that $x \mapsto \hat{a}(x,v(x))$ is bounded. Then, for each $x \in \R^d$, we have 
\begin{align*}
    \inf_{\alpha \in \sA} J_x(\alpha) = \lambda = J_x(\alpha^*), 
\end{align*}
where $\alpha^*  = \hat{a}(X^x, v(X^x)) \in \sA$. 
\end{prop}

\subsubsection{Conditions for existence of an optimal control}

Now we establish conditions on the running cost $r$ under which the ergodic BSDE \eqref{ebsdecontrol} is covered by Theorem \ref{thm:main}. The main difficulty is to check what conditions on $r$ will force the driver $H(x,z) = h(x,z,\hat{a}(x,z))$ to satisfy the monotonicity condition 
\begin{align} \label{ergmonotone}
    \big(\partial_x H(x,z) - \partial_x H(x,\bz)\big) \cdot (z - \bz) \leq 0.
\end{align}

We start with a Lemma, which is proved by differentiating the optimality condition
\begin{align} \label{optcondition}
    z^i + \partial_a r^i(x, \hat{a}(x,z)) = 0. 
\end{align}
implicitly in $z$.
\begin{lem} \label{lem:fderivcomp}
If $r$ is $C^2$, then
\begin{align*}
    \partial_z \partial_x H(x,z) = - \partial_{a} \partial_x r(x, \hat{a}(x,z)) (\partial_a^2 r(x,\hat{a}(x,z))^{-1}.
\end{align*} In particular, by Proposition \ref{prop:dissipative}, $H$ satisfies the monotonicity condition \eqref{monotone} whenever $\partial_a \partial_x r(x,a)\partial_a^2 r(x,a))^{-1} $ is non-negative, i.e.
\begin{align}
    z^T\partial_a \partial_x r(x,a)\partial_a^2 r(x,z)^{-1} z \geq 0
\end{align}
holds for all $x,z,a \in \R^d$. 
\end{lem}

\begin{remark}
Notice that in the case $d = 1$, the condition that $\partial_a \partial_x r (\partial_a^2 r)^{-1}$ is non-negative is equivalent to the condition that $\partial_a \partial_x r \geq 0$, since we have already assumed that $\partial_x^2 r > 0$. Notice also that in any dimension, the condition is satisfied as soon as $\partial_a \partial_x r = 0$, i.e. if $r$ has a separated structure 
\begin{align*}
    r(x,a) = r_1(x) + r_2(a). 
\end{align*}
\end{remark}

\begin{prop} \label{prop:controlexist}
Suppose that in addition to Assumption \ref{assump:ergcontrol}, $r$ is $C^2$. Suppose further $\partial_a \partial_x r \big( \partial_a^2 r\big)^{-1}$ is non-negative. Finally, assume that $H$ is bounded and Lipschitz on $\R^d \times B_R$ for each $R > 0$. 
Then the drift $b$ and the driver $H$ satisfy Assumption \ref{assump:existence}, and so by Theorem \ref{thm:main} the ergodic BSDE \eqref{ebsdecontrol} has a unique Markovian solution $(u,v,\lambda)$ with $u$ of linear growth and $v$ bounded.
\end{prop}

\subsection{Risk-sensitive ergodic control} \label{subsec:risk}
Now we apply the reasoning from the previous sub-section to a risk-sensitive ergodic control problem, as introduced in \cite{Fleming1995RiskSensitiveCO}. As in the previous sub-section, the data of the control problem consists of a vector field $b : \R^d \to \R^d$ and a cost function $r : \R^d \times \R^d \to \R$. We will again enforce Assumption \ref{assump:ergcontrol}, and define $\sA$, $\pr^{\alpha,T}$, $H$, $\hat{a}$, etc. in the same way. But this time we fix $\delta > 0$, and define the cost functionals $J_x^{\delta} : \sA \to \R$ by 
\begin{align}
    J_x^{\delta}(\alpha) = \limsup_{T \to \infty} \frac{1}{T} \big( \frac{1}{\delta} \ln \E^{\pr^{\alpha,T}}[\exp(\delta \int_0^T r(X_t^x, \alpha_t)] dt \big).
\end{align}
It turns out that this control problem is related to the ergodic BSDE 
\begin{align} \label{ebsderisk}
    dY_t = \lambda dt - f(X_t, Z_t) dt + Z_t \cdot dW_t,\\
    \text{with } \nonumber
    f(x,z) = H(x,z) + \frac{\delta}{2} |z|^2, 
\end{align}
where the Hamiltonian $H$ is as defined in the previous sub-section. 
We now state the connection between the risk-sensitive control problem and the ergodic BSDE \eqref{ebsderisk}. 
\begin{prop} \label{prop:ebsdeoptimalityrisk}
Suppose that the ergodic BSDE \eqref{ebsderisk} has a Markovian solution $(u,v,\lambda)$ such that $u$ is of linear growth and $v$ is bounded, and that $x \mapsto \hat{a}(x,v(x))$ is bounded. Then, for each $x \in \R^d$, we have 
\begin{align*}
    \inf_{\alpha \in \sA} J_x^{\delta}(\alpha) = \lambda = J_x^{\delta}(\alpha^*), 
\end{align*}
where $\alpha^*  = \hat{a}(X^x, v(X^x)) \in \sA$. 
\end{prop}
This result is a fairly standard generalization of Proposition \ref{prop:ebsdeoptimality}, and the proof of Proposition 4 in \cite{Liang2017RepresentationOH} applies directly to this slightly more general setting. Thus we omit the proof. 
Because the driver $f$ in \eqref{ebsderisk} differs from the driver $H$ appearing in the previous sub-section only by the term $\frac{\delta}{2} |z|^2$, it is clear that $f$ satisfies Assumption \ref{assump:existence} if and only if $H$ does. Thus we get the following analogue of Proposition \ref{prop:controlexist}.

\begin{prop}
Suppose that in addition to Assumption \ref{assump:ergcontrol}, $r$ is $C^2$. Suppose further that $\partial_a \partial_x r \big( \partial_a^2 r\big)^{-1}$ is non-negative. Finally, assume that $H$ is bounded and Lipschitz on $\R^d \times B_R$ for each $R > 0$. 
Then the drift $b$ and the driver $f$ satisfy Assumption \ref{assump:existence}, and so by Theorem \ref{thm:main} the EBSDE \eqref{ebsderisk} has a unique Markovian solution $(u,v,\lambda)$ with $u$ of linear growth and $v$ bounded.
\end{prop}

\subsection{Weak vs. strong formulation} \label{subsec:weakstrong}

The control problems discussed in the previous two subsections are in weak formulation, in the sense that the underlying diffusion $X^x$ is fixed and the controller affects the dynamics of $X^x$ by changing the probability measure. This formulation is often easier to work with than the strong formulation, but it is interesting to note that because we are in the Markovian setting and the controls produced are bounded, we can treat the strong formulation by the same techniques.
\newline \newline
As in the weak formulation, we fix a vector field $b : \R^d \to \R^d$ and a cost function $r : \R^d \times \R^d \to \R$ satisfying Assumption \ref{assump:ergcontrol}. Then we define $\sA^S$ to be the set of all bounded and measurable maps $\alpha = \alpha(t,x) : \R_{+} \times \R^d \to \R^d$, and note that by a classical result which can be traced to Veretennikov (see the introduction of \cite{ZHANG20051805} for a modern account of existence of strong solutions for SDEs with irregular drift), for each $x \in \R^d$ and $\alpha \in \sA^S$, there is a unique solution $X^{x,\alpha}$ to the SDE 
\begin{align*}
    \begin{cases}
    dX_t^{x,\alpha} = \big(b(X_t^{x,\alpha}) + \alpha(t,X_t^{x,\alpha}) \big) dt + dW_t, \\
    X^{x,\alpha} = x. 
    \end{cases}
\end{align*}
For each $x \in \R^d$, we define $J_x^S : \sA^S \to \R$ by 
\begin{align*}
    J_x^S(\alpha) = \limsup_{T \to \infty} \frac{1}{T} \E[\int_0^T r(X_t^{x,\alpha}, \alpha(t,X_t^{x,\alpha}))dt].
\end{align*}
This formulation will be equivalent to the weak formulation in the following sense. 
\begin{prop}
Suppose that the ergodic BSDE \eqref{ebsdecontrol} has a Markovian solution $(u,v,\lambda)$ such that $u$ has linear growth, $v$ is bounded and $\alpha^*(t,x) = \alpha^*(x) \coloneqq \hat{a}(x,v(x))$ is bounded. Then 
\begin{align*}
    \inf_{\alpha \in \sA^{S}} J_x^S(\alpha) = J_x^S(\alpha^*) = \lambda. 
\end{align*}
In particular, $\inf_{\alpha \in \sA^{S}} J_x^S(\alpha) = \inf_{\alpha \in \sA} J_{x}(\alpha)$. 
\end{prop}

\begin{proof}
Let $(u,v)$ be the Markovian solution to \eqref{ebsdecontrol} which is supposed to exist, and for fixed $x \in \R^d$ set $(Y^x,Z^x) = (u(X^x), v(X^x))$, where $X^x = X^{x,\alpha*}$. We claim that
\begin{align} \label{markov}
    Y_t^x = Y_T^x - \lambda (T-t) + \int_t^T r(X^x_s,\hat{a}(X^x_s,Z^x_s)) ds - \int_t^T Z^x_s \cdot dW_s. 
\end{align}
Indeed, the fact that $(u,v)$ is a Markovian solution for \eqref{ebsdecontrol} implies that if $X = X^0$, and $(Y,Z) = (u(X),v(X))$ then we have
\begin{align*}
    Y_t = Y_T - \lambda(T -t) + \int_t^T h(X_s,Z_s,\hat{a}(X_s,Z_s)) ds - \int_t^T Z_s \cdot dW_s \\
    = Y_T - \lambda(T-t) + \int_t^T  \Big(h(X_s,Z_s,\hat{a}(X_s,Z_s)) + b(X_s) \cdot v(X_s)\Big)  ds - \int_t^T Z_s \cdot dX_s,
\end{align*}
i.e. 
\begin{align*}
    u(X_t) = u(X_T) - \lambda(T-t) + \int_t^T \Big(h(X_s,v(X_s),\hat{a}(X_s,v(X_s))) + b(X_s) \cdot v(X_t)) \Big) ds \\ - \int_t^T v(X_s) \cdot dX_s. 
\end{align*}
But for each $T > 0$ there is an equivalent probability measure $\Q^T$ under which the law of $X^x$ on $[0,T]$ is equal to the law of $X$ under $\pr$, and so the equation
\begin{align} \label{comcomp}
    u(X_t^x) = u(X_T^x) - \lambda(T-t) + \int_t^T 
    \Big(h(X_s^x,v(X_s^x),\hat{a}(X_s^x,v(X_s^x))) + b(X_s^x) \cdot v(X_s^x) \
    \Big) ds \nonumber \\ - \int_t^T v(X_s^x) \cdot dX_s^x
\end{align}
holds under $\Q^T$, hence also under $\pr$. Combining \eqref{comcomp} with the identities
\begin{align*}
    dX^x_s = \big(b(X^x_s) + \hat{a}(X^x_s, v(X^x_s))\big) + dW_s, \\
    h(X_s^x,v(X_s^x),\hat{a}(X_s^x,v(X_s^x))) = \hat{a}(X_s^x,v(X_s^x)) \cdot v(X_s^x) + r(X_s^x,\hat{a}(X_s^x,v(X_s^x))), 
\end{align*}
we arrive at 
\begin{align*}
    u(X_t^x) = u(X_T^x) - \lambda(T-t) + \int_t^T 
    r(X_s^x,\hat{a}(X_s^x,v(X_s^x))) ds - \int_t^T v(X_s^x) \cdot dW_s.
\end{align*}
Recalling that $(Y^x,Z^x) = (u(X^x),v(X^x))$, we get \eqref{markov}.  From \eqref{markov}, we in turn get
\begin{align*}
    \lambda = \frac{1}{T} \Big(Y_T^x - Y_0^x + \int_0^T r(X_s^x, \hat{a}(X_s^x,Z^x_s)) dt - \int_0^T Z_s^x \cdot dW_s \Big)
\end{align*}
For any fixed $\alpha \in \sA^S$, we can further manipulate the expression for $\lambda$ to read 
\begin{align} \label{lambda}
    \lambda = \frac{1}{T} \Big(Y_T^x - Y_0^x + \int_0^T \big( h(X_s^x,Z_s^x,\hat{a}(X_s^x, Z_s^x)) - h(X_s^x,Z_s^x,\alpha(s,X_s^x)) ds \nonumber  \\ + \int_0^T r(X_s^x,\alpha(s,X_s^x)) ds - \int_0^T Z_s^x \cdot dW_s^{\alpha} \Big), 
\end{align}
where $W^{\alpha} = W - \int \alpha(\cdot,X^x) dt + \int \alpha^*(X^x) dt$ is a Brownian motion on $[0,T]$ under $\pr^{\alpha,T}$, where $d\pr^{\alpha,T} = \mathcal{E}(\int \big[ \alpha(\cdot,X^x) - \alpha^*(X^x) \big] \cdot dW)_T$. We note that we have 
\begin{align*}
    dX_t^{x} = \big(b(X_t) + \alpha(t,X_t)\big) dt + dW^{\alpha}_t
\end{align*}
under $\pr^{\alpha,T}$, i.e. $X^x$ is a weak solution of the same SDE which defines $X^{x,\alpha}$. By uniqueness in law for this equation, we deduce that the law of $X^{x}$ under $\pr^{\alpha,T}$ is the same as that of $X^{x,\alpha}$ under $\pr$, and in particular from \eqref{lambda} we get 
\begin{align*}
    \lambda = \frac{1}{T} \E^{\pr^{\alpha, T}}[Y_T^x - Y_0^x + \int_0^T \big( h(X_s^x,Z_s^x,\hat{a}(X_s^x, Z_s^x)) - h(X_s^x,Z_s^x,\alpha(s,X_s^x)) \big) ds ] \\
    + \frac{1}{T}\E[\int_0^T r(X_s^{x,\alpha},\alpha(s,X_s^{x,\alpha})) ds ]. 
\end{align*}
Since $\big( h(X_s^x,Z_s^x,\hat{a}(X_s^x, Z_s^x)) - h(X_s^x,Z_s^x,\alpha(s,X_s^x)) \big) \leq 0$, we can take $T \to \infty$ to get $\lambda \leq J_x^S(\alpha)$. Similar arguments show that $\lambda = J_x^S(\alpha^*)$, and this completes the proof. 
\end{proof}

\section{Equation from forward performance processes} \label{sec:forward}
In this section, we consider the driver 
\begin{align} \label{forwarddriver}
    f(x,z) = \frac{1}{2} \delta^2 \text{dist}^2(\Pi, \frac{\theta(x) + z}{\delta}) - z^T \theta(x) + \frac{1}{2} |\theta(x)|^2, 
\end{align}
where $\Pi$ is a convex and closed subset of $\R^d$, and $\theta : \R^d \to \R^d$. This driver is used in Section 4 of \cite{Liang2017RepresentationOH} to construct forward performance processes in factor form. In their context, the diffusion $X$ is a stochastic factor process, and $\theta = \theta(x)$ is the market price of risk corresponding to the vector of stock prices $S$, which solves 
\begin{align*}
     dS_t^i = S_t^i \big( \eta^i(X_t) dt + \sum_{j = 1}^d \sigma^{ij}(X_t) dW_t^j \big)
\end{align*}
for $\eta$ and $\sigma$ satisfying appropriate conditions. That is, $
    \theta  = \sigma^T [\sigma \sigma^T]^{-1} \eta.$
    \newline \newline
Because we are primarily interested in interpreting our monotonicity condition \eqref{monotone} in this setting, we refer to \cite{Liang2017RepresentationOH} for more details on forward performance processes and their connection to the driver \eqref{forwarddriver}. Here, we simply take the market price of risk $\theta : \R^d \to \R^d$ and the convex closed set $\Pi \subset \R^d$ as given, and we focus on checking what conditions on $\theta$ and $\Pi$ will guarantee that $f$ satisfies the hypotheses of Theorem \ref{thm:main}. We will see that in a few simple cases the monotonicity condition \eqref{monotone} boils down to a monotonicity condition on $\theta$. As usual, we fix a vector field $b : \R^d \to \R^d$ which satisfies the first condition in Assumption \ref{assump:existence}. 

Let us suppose that $h = 0$ in the decomposition $f = g+ h$. Then we need to check under what conditions on $\Pi$ and $\theta$ we have the monotonicity condition 
\begin{align} \label{forwardmonotone}
    \big(\partial_x f(x,z) - \partial_x f(x,\bz)\big) \cdot (z - \bz) \leq 0. 
\end{align}
First, we recall that since $\Pi$ is closed and convex, there is a well-defined projection $P_{\Pi} : \R^d \to \Pi$, which maps $x \in \Pi^c$ to the unique element of $\text{argmin}_{p \in \Pi} |x - p|$. Moreover, $x \mapsto \text{dist}(\Pi,x)$ is Lipchitz, and its weak gradient $\nabla \text{dist}(\Pi,\cdot)$ is given by
\begin{align*}
    \nabla \text{dist}(\Pi,\cdot)(x) = \begin{cases}
    \frac{x - \Pi(x)}{|x - \Pi(x)|} & x \in \Pi^c \\
    0 & x \in \Pi. 
    \end{cases}
\end{align*}
Using this, we find that
\begin{align*}
    \nabla \text{dist}^2(\Pi,\cdot)(x) = \begin{cases}
    2 (x - P_{\Pi}(x)) & x \in \Pi^c, \\
    0 & x \in \Pi.
    \end{cases}
\end{align*}
Now we can compute $\partial_x f$ where $f$ is the driver \eqref{forwarddriver}. We find that 
\begin{align*}
    [\partial_x f(x,z)]^T = \delta \big(\frac{\theta(x) + z}{\delta} - P_{\Pi}(\frac{\theta(x) + z}{\delta})\big)D \theta(x) - z^TD \theta(x) + \theta(x)^T D \theta(x) \\
    = - \delta P_{\Pi}(\frac{\theta(x) + z}{\delta}) D \theta(x) + 2\theta(x)^T D\theta(x). 
\end{align*}
Using this we check that the condition \eqref{forwardmonotone} is equivalent to the condition 
\begin{align} \label{equivmonotone}
    \Big(P_{\Pi}\big( \frac{\theta(x) + z}{\delta}\big) - P_{\Pi}\big(\frac{\theta(x) + \bz}{\delta}\big) \Big)D \theta(x) (z - \bz) \geq 0. 
\end{align}
From this, we easily deduce the following Proposition:
\begin{prop} \label{prop:forward}
Suppose that $\theta$ is Lipschitz and bounded, and that the condition \eqref{equivmonotone} holds. Then $b$ and $f$ satisfy Assumption \ref{assump:existence}, and so the ergodic BSDE \eqref{ebsde} admits a unique Markovian solution by Theorem \ref{thm:main}.
\end{prop}

\begin{remark}
Some computations reveal that the condition \eqref{equivmonotone} holds (and hence Proposition \ref{prop:forward} applies) if
\begin{itemize}
    \item $d = 1$ and $\theta' \geq 0$, \text{ or }
    \item $\Pi = \R^d$ and $D \theta \geq 0$, in the sense that $z^T D \theta(x) z \geq 0$ for all $x,z \in \R^d$.
\end{itemize}
Thus in these two simple cases, the monotonicity condition \eqref{monotone} on the driver $f$ reduces to a certain monotonicity condition on the market price of risk vector $\theta$. Unfortunately, this does not hold in general, i.e. it is possible to choose $\Pi$ and $\theta$ such that $D \theta$ is non-negative, but \eqref{equivmonotone} is not satisfied.
\end{remark}

\appendix
\section{Some results from elliptic regularity theory}

In this appendix, we establish rigorously the connection between the discounted BSDE 
\begin{align} \label{dbsde}
\begin{cases}
dX_t = b(X_t) dt + dW_t, \\
    dY_t = \rho Y_t dt - f(X_t,Z_t) dt + Z_t \cdot dW_t
    \end{cases}
\end{align}
and the PDE \eqref{dpde}
\begin{align} \label{dpde}
\frac{1}{2} \Delta u^{\rho} + b \cdot \nabla u^{\rho} + f(x, \nabla u^{\rho}) = \rho u^{\rho},
\end{align}
and also the connection between the ergodic BSDE \eqref{ebsde} and the PDE \eqref{pde}. 
We start with a preliminary Lemma which is standard in the ergodic BSDE literature. 

\begin{lem} \label{lem:prelim}
Suppose that in addition to Assumption \ref{assump:existence}, $f$ is Lipschitz, and let $(u,v,\lambda)$ be the unique solution to the ergodic BSDE \eqref{ebsde}. Then we have 
\begin{itemize}
    \item $v = \nabla u$ a.s.
    \item there exists $x_0 \in \R^d$ and a sequence $\rho_n \downarrow 0$ such that $u$ is locally uniform limit of the sequence $u^{\rho_n} - u^{\rho_n}(x_0)$, and $\lambda$ is limit of $\rho_nu^{\rho_n}(x_0)$, where $(u^{\rho_k}, v^{\rho_k})$ is the Markovian solution of the discounted BSDE \eqref{dbsde} with $\rho = \rho_k$.
\end{itemize}
\end{lem}

The first assertion of Lemma \ref{lem:prelim} is a special case of Theorem 5.3 in \cite{Fuhrman2009ErgodicBA}, while the second assertion is clear from e.g. the appendix of \cite{Liang2017RepresentationOH}. 

The next lemma establishes existence and regularity for a certain semi-linear elliptic equation. 

\begin{lem}[Semi-linear Dirichlet problem] \label{lem:dirichlet}
Suppose that in addition to Assumption \ref{assump:existence}, $f$ is Lipschitz and bounded. For each $R > 0$ and $\phi : \partial B_R \to \R$ Lipschitz, there exists a unique Lipschitz classical solution $u$ to the Dirichlet problem 
\begin{align} \label{dirichlet}
    \begin{cases} 
    \frac{1}{2} \Delta u + b \cdot \nabla u + f(x,\nabla u) = \rho u \text{ in } B_R \\
    u = \phi \text{ on } \partial B_R. 
    \end{cases}
\end{align}
Moreover, if $b$ and $f$ are $C^{k,\alpha}_{loc}$, then $u$ is in $C^{k+2}$ in $B_R$.  
\end{lem}

\begin{proof}
We proceed in three steps. \newline \newline
\textit{Step 1 (reduction to zero boundary condition):} We begin by posing the linear problem 
\begin{align*}
    \begin{cases} 
    \frac{1}{2} \Delta \tilde{u} + b \cdot \nabla \tilde{u} = \rho \tilde{u} \text{ in } B_R \\
    \tilde{u} = \phi \text{ on } \partial B_R. 
    \end{cases}
\end{align*}
Since $\phi$ and $b$ are Lipschitz, existence and regularity results for linear elliptic equations give a unique classical solution $\tilde{u}$ which is Lipschitz on $B_R$. Moreover, if $b \in C^{k,\alpha}_{loc}$, then $\tilde{u}$ lies is $C^{k,\alpha}_{loc}(B_{R})$, thanks to the interior Schauder estimates. Subtracting the equations for $u$ and $\tilde{u}$ gives an equation for $w \coloneqq u - \tilde{u}$, namely
\begin{align} \label{shiftedeqn}
    \begin{cases} 
    \frac{1}{2} \Delta w + b \cdot \nabla w + g(x,\nabla w) = \rho w \text{ in } B_R \\
    w = 0 \text{ on } \partial B_R,
    \end{cases}
\end{align}
where $g(x,z) = f(x,z + \nabla \tilde{u}(x)) : B_R \times \R^d \to \R$. We note that $g$ is uniformly Lipschitz in $z$, and locally Lipschitz in $x$.
\newline \newline 
\textit{Step 2 (Existence of a weak solution):} Now we focus on producing a solution in $H_0^1$ to the equation \eqref{shiftedeqn}. We define a map $\Phi : H_0^1(B_R) \to H_0^1(B_R)$ by $\Phi(\tilde{w}) = w$, where $w$ is the unique solution (which exists by Theorem 8.3 of \cite{Gilbarg1977EllipticPD}) to the linear problem \begin{align} 
    \begin{cases} 
    \frac{1}{2} \Delta w + b \cdot \nabla w + g(x,\nabla \tilde{w}) = \rho w \text{ in } B_R \\
    w = \phi \text{ on } \partial B_R. 
    \end{cases}
\end{align} 
Notice $g$ is uniformly Lipschitz in $z$, so that for $\tilde{w}, \tilde{w}' \in H_0^1(B_R)$, we have 
\begin{align*}
    \norm{ g(\cdot, \nabla \tilde{w}) - g(\cdot, \nabla \tilde{w}')}_{\ltwo(B_R)} \leq C \norm{ \nabla \tilde{w} - \nabla \tilde{w}' }_{\ltwo(B_R)} \leq C \norm{\tilde{w} - \tilde{w}'}_{H_0^1(B_R)}
\end{align*}
for some constant $C$, and thus 
by Corollary 8.7 of \cite{Gilbarg1977EllipticPD}
\begin{align*}
    \norm{\Phi(\tilde{w}) - \Phi(\tilde{w}')}_{H_0^1(B_R)} \leq C \norm{\tilde{w} - \tilde{w}'}_{H_0^1(B_R)}.
\end{align*}
That is, $\Phi : H_0^1(B_R) \to H_0^1(B_R)$ is continuous. Furthermore, by Calderon-Zygmund estimates (see e.g. Theorem 8.10 of \cite{Gilbarg1977EllipticPD}), the image of $\Phi$ is contained in a bounded subset of $H^2(\Omega)$, and in particular is compact in $H_0^1$. Thus by the Schauder fixed point theorem, there exists a fixed point of $\Phi$, i.e. a solution of \eqref{shiftedeqn}. 
\newline \newline 
\textit{Step 3 (bootstrapping to higher regularity):} Now that we have a solution $w$ to \eqref{shiftedeqn}, we note that (again by Calderon-Zygmund estimates) we have that $w \in W^{2,p}(B_1)$ for all $p > 1$, hence in $C^{1,\beta}$ for some $\beta$ by Sobolev embedding. Since $g$ is locally Lipschitz, it follows that $x \mapsto g(x,\nabla w(x))$ is actually $C^{\beta}_{loc}$, so in fact by local Schauder estimates for linear elliptic equations (see Theorem 6.6 and Corollary 6.9 of \cite{Gilbarg1977EllipticPD})) we have $w \in C^{2,\beta}(B_R)$. If $b,f \in C^{1,\alpha}_{loc}$, then we see that $g \in C^{1,\alpha}_{loc}(B_R)$. Since $w \in C^{2,\beta}(B_R)$, we deduce that $x \mapsto g(x, \nabla w(x))$ is in $C^{1,\gamma}_{loc}(B_R)$ for some $\gamma$. Thus applying again the interior Schauder estimates (see Theorem 6.17 of \cite{Gilbarg1977EllipticPD}) we have $w \in C^{3,\gamma}_{loc}(B_R)$. If $b,f$ are $C^{k,\alpha}_{loc}$, we just continue this bootstrapping argument to find that in fact $u \in C^{k+2,\gamma}_{loc}(B_R)$ for some $\gamma$, which completes the proof of existence. 
\newline \newline
\textit{Step 4 (uniqueness):}
For uniqueness, one checks that if $u$ is a Lipschitz classical solution to \eqref{dirichlet}, then for each $x \in \R^d$, $(Y^x,Z^x) = (u(X^x), \nabla u(X^x))$ solves the BSDE 
\begin{align*}
    \begin{cases}
    dY_t = \rho Y_t dt - f(X_t,Z_t) dt + Z_t \cdot dW_t, \\
    Y_{\tau} = g(X^x_{\tau})
    \end{cases}
\end{align*}
on $[0,\tau]$, where $\tau$ is the stopping time $\tau = \inf \{t \geq 0 : |X^x_t| = R\}$. Thus uniqueness for the PDE \eqref{dirichlet} follows from uniqueness for such BSDEs, see e.g. Theorem 3.3 of \cite{BRIAND1998455}. 
\end{proof}

\begin{lem}[Regularity of Markovian solutions to discounted BSDE] \label{lem:regularity}
Suppose that in addition to Assumption \ref{assump:existence}, $f$ is globally Lipschitz, and let $(u^{\rho},v^{\rho})$ be the unique Markovian solution to the discounted BSDE \eqref{dbsde}. Then in fact $u^{\rho}$ is a classical solution to the PDE \eqref{dpde}, and $v^{\rho} = \nabla u^{\rho}$ almost surely. Moreover, if $b$ and $f$ are $C^{k,\alpha}_{loc}$, then $u^{\rho}$ is $C^{k+2}$.
\end{lem}


\begin{proof}
For $R > 0$, consider the Dirichlet problem 
\begin{align} \label{deqn}
    \begin{cases}
    \frac{1}{2} \Delta \tilde{u} + b \cdot \nabla \tilde{u}^{\delta} + f(x, \nabla \tilde{u}^{\delta}) = \rho \tilde{u}^{\delta} \text{ in } B_R, \\
    \tilde{u}^{\rho} = u^{\rho} \text{ on } \partial B_R.  
    \end{cases}
\end{align}
First, assume also that $f$ is bounded. By Lemma \ref{lem:dirichlet}, we have a unique Lipschitz classical solution $\tilde{u}$.
Now we fix $x \in B_R$, and we define 
\begin{align*}
    (Y^x,Z^x) = (u^{\delta}(X^{x,\tau}), v^{\delta}(X^{x,\tau})), \\  (\tilde{Y}^x,\tilde{Z}^x) = (\tilde{u}^{\delta}(X^{x,\tau}), \nabla \tilde{u}^{\delta}(X^{x,\tau}))
\end{align*} where $X^{x,\tau}$ denotes the process $X^x$ stopped at $\tau \coloneqq \inf \{t \geq 0 : |X_t^x| = R\}$. One can check that $(\tilde{Y}^x,\tilde{Z}^x)$ solves the BSDE
\begin{align} \label{auxeqn}
    \begin{cases} 
    dY_t = \big[\rho Y_t - f(X_t, Z_t) dt + Z_t \cdot dW_t\big]1_{t \leq \tau}, \\
    Y_{\tau} = u^{\delta}(X^x_{\tau})
    \end{cases}
\end{align}
on the stochastic interval $[0,\tau]$. 
Of course, $(Y,Z)$ also satisfies \eqref{auxeqn}, so by uniqueness for (see \cite{BRIAND1998455} Theorem 3.3) we must have 
\begin{align*}
    u^{\delta}(x) = Y^x_0 = \tilde{Y}^x_0 = \tilde{u}^{\delta}(x), 
\end{align*}
and $v^{\delta} = \nabla \tilde{u}$ almost surely. Lemma \ref{lem:dirichlet} also shows that $u$ is $C^{k+1}$ as soon as $b$ and $f$ are $C^{k,\alpha}_{loc}$. Finally, to remove the hypothesis that $f$ is bounded, we consider the drivers $f^j(x,z) = f(x,\pi^j(z))$, where $\pi^j$ is a sequence of maps $\R^d \to \R^d$ which each smooth, have Lipschitz constant $1$, are bounded, and satisfy $\pi^j(z) = z$ for $|z| \leq j$. Then let $(u^{j},v^j)$ be the Markovian solution to the discounted BSDE 
\begin{align*}
    dY_t = \rho Y_t - f^j(X_t, Z_t) dt + Z_t \cdot dW_t, 
\end{align*}
and note that we have already shown that $u^j$ is a $C^{k+2}$ solution of 
\begin{align*}
    \frac{1}{2} \Delta u^j + b \cdot \nabla u^j + f^j(x,\nabla u^j) = \rho u^j
\end{align*}
and that $v^j = \nabla u^j$. Finally, we apply a-priori estimates for Lipschitz discounted BSDEs to get that $\sup_j \norm{v^j}_{\linf} < \infty$, and so in fact for large enough $j$ $(u^j,v^j) = (u,v)$, where $(u,v)$ is the unique Markovian solution to \eqref{dbsde}. This completes the proof. 
\end{proof}

\begin{prop}[Regularity for Markovian solutions of ergodic BSDE] \label{prop:regularityebsde}
Suppose that in addition to Assumption \ref{assump:existence}, $f$ is globally Lipschitz, and let $(u,v)$ be the unique Markovian solution to the ergodic BSDE \eqref{ebsde}. Then $u$ is in fact a classical solution to the PDE \eqref{pde}, and $v = \nabla u$ a.e. Moreover, if $b$ and $f$ are $C^{k,\alpha}_{loc}$, then $u$ is $C^{k+2}$. 
\end{prop}

\begin{proof}
We know that $u$ is the locally uniform limit of $\tilde{u}^{\rho_n} \coloneqq u^{\rho_n} - u^{\rho_n}(0)$, where for $\rho > 0$, $(u^{\rho}, v^{\rho})$ is the unique Markovian solution of the discounted BSDE \eqref{dbsde}. By Lemma \ref{lem:regularity}, $\tilde{u}^{\rho_n}$ is in fact a smooth solution of the PDE 
\begin{align*}
    \frac{1}{2} \Delta \tilde{u}^{\rho_n} + b \cdot \nabla \tilde{u}^{\rho_n} + f(x, \nabla \tilde{u}^{\rho_n}) = \rho_n \tilde{u}^{\rho_n} + \rho_n u^{\rho_n}(0) \text{ in } \R^d. 
\end{align*}
Let us rewrite this equation as 
\begin{align} \label{simplerpde}
    \sL \tilde{u}^{\rho_n} = F^n(x) \text{ in } \R^d,
    \end{align}
    where 
    \begin{align*}
    F^n(x) = \rho_n \tilde{u}^{\rho_n} + \rho_n u^{\rho_n}(0) - f(x, \nabla \tilde{u}^{\rho_n}),\,\,
    \sL \tilde{u}^{\rho_n} = \frac{1}{2} \Delta \tilde{u}^{\rho_n} + b \cdot \nabla \tilde{u}^{\rho_n}. 
\end{align*}
By Lemma 3.1 of \cite{BRIAND1998455} (see also the Appendix of \cite{Liang2017RepresentationOH}), we have \begin{align}
    \norm{u^{\rho_n}}_{\linf} \leq \frac{\sup_{x \in \R^d} f(x,0)}{\rho_n},
\end{align} and in particular the sequence of real number $\rho_n u^{\rho_n}(0)$ is bounded. Moreover, since we know that $u^{\rho_n}$ is Lipschitz uniformly in $n$ (again see the Appendix of \cite{Liang2017RepresentationOH}), we observe  that for each $R > 0$,
\begin{itemize}
    \item $\norm{f(\cdot, \nabla \tilde{u}^{\rho_n})}_{\linf(\R^d)}$ is bounded uniformly in $n$,
    \item $\norm{\rho_n \tilde{u}^{\rho_n}}_{\linf(B_{2R})}$ converges to zero,
    \item $\rho_n u^{\rho_n}(0)$ is bounded uniformly in $n$. 
\end{itemize} Putting these observations together, we see that for each $R > 0$, $\norm{F^n}_{\linf(B_{2R})}$ is bounded uniformly in $n$. By the interior Calderon-Zygmund estimates, we have for each $R > 0$ and any $1 < p < \infty$ the estimate 
\begin{align*}
    \norm{\tilde{u}^{\rho_n}}_{W^{2,p}(B_R)} \leq C \big(\norm{\tilde{u}^{\rho_n}}_{L^{p}(B_{2R})} + \norm{F^n}_{L^p(B_{2R})}\big), 
\end{align*}
for some $C$ independent of $n$, and in particular we see that for each $1 < p < \infty$ and each $R > 0$, $\norm{\tilde{u}^{\rho_n}}_{W^{2,p}(B_R)} $ is bounded uniformly in $n$. By choosing $p$ large enough and applying Sobolev embedding, we find that in fact $\norm{\tilde{u}^{\rho_n}}_{C^{1,\alpha}(B_R)}$ is bounded uniformly in $n$ for some $0 < \beta < 1$. But now similar reasoning as above shows that $\norm{F^n}_{C^{1,\beta}(B_{2R})}$ is bounded uniformly in $n$. Now from the interior Schauder estimates we get 
\begin{align*}
    \norm{\tilde{u}^{\rho_n}}_{C^{2,\beta}(B_R)} \leq C \big(\norm{\tilde{u}^{\rho_n}}_{C^{\beta}(B_{2R})} + \norm{F^n}_{C^{\beta}(B_{2R})}\big)
\end{align*}
for some $C$ independent of $n$, and in particular we find that $\norm{\tilde{u}^{\rho_n}}_{C^{2,\beta}(B_R)}$ is bounded uniformly in $n$. Since $C^{2,\beta}(B_R)$ compactly embeds in $C^{2,\gamma}$ when $0 < \gamma < \beta < 1$, we find that (passing to a further subsequence if necessary), we in fact have 
\begin{align*}
    \norm{u - \tilde{u}^{\rho_n}}_{C^{2,\gamma}(B_R)} \to 0 \text{ as } n \to \infty
\end{align*}
for some $0 < \gamma < \beta$ and all $R > 0$.
From Lemma \ref{lem:prelim}, it is now easy to see that $u$ is a classical solution of \eqref{pde}. To show that $u$ is $C^{k+2}$ when $b$ and $f$ are $C^{k,\alpha}_{loc}$, we continue the above bootstrapping argument to show that if $b$ and $f$ are $C^{k,\alpha}_{loc}$, then in fact $\norm{\tilde{u}^{\rho_n}}_{C^{(k+2),\gamma}}$ is bounded uniformly in $n$ for some $\gamma$, and then by the same reasoning as above 
\begin{align*}
    \norm{u - \tilde{u}^{\rho_n}}_{C^{k+2,\gamma}(B_R)} \to 0 \text{ as } n \to \infty,
\end{align*}
for each $R > 0$, and so in particular $u \in C^{k+2,\gamma}_{loc}(\R^d)$. Finally, by Lemma \ref{lem:prelim}, we have $v = \nabla u$ a.e.
\end{proof}

\bibliographystyle{amsplain}
\bibliography{ebsde}
\end{document}